\documentclass[11pt,oneside]{amsart}
\usepackage[top=25mm, bottom=25mm, left=25mm, right=25mm]{geometry}
\usepackage{amsfonts,amssymb,amsthm,amsxtra}
\usepackage[T1]{fontenc}
\usepackage[utf8]{inputenc}
\usepackage{bm}
\usepackage{mathtools}
\usepackage{dsfont}
\usepackage{bbm}
\usepackage{mathrsfs}
\usepackage{enumitem}
\usepackage{mathabx}
\usepackage{graphics}

\usepackage{graphics}

\usepackage{color}

\numberwithin{equation}{section}
\newtheorem{theorem}[equation]{Theorem}
\newtheorem{corollary}[equation]{Corollary}

\newtheorem{proposition}[equation]{Proposition}
\newtheorem{fact}[equation]{Fact}
\theoremstyle{definition}

\newtheorem{definition}[equation]{Definition}

\newtheorem{conj}[equation]{Conjecture}

\newcommand{\DD}{\mathbb{D}}

\newcommand{\II}{\mathbb{I}}

\newcommand{\NN}{\mathbb{N}}

\newcommand{\RR}{\mathbb{R}}

\newcommand{\ZZ}{\mathbb{Z}}

\newcommand{\calC}{\mathcal{C}}

\newcommand{\calF}{\mathcal{F}}

\newcommand{\calM}{\mathcal{M}}

\newcommand{\calH}{\mathcal{H}}

\newcommand{\calS}{\mathcal{S}}
\newcommand{\calD}{\mathcal{D}}

 \newcommand{\sgn}{\operatorname{sgn}}

\newcommand*{\DMO}[1]{\expandafter\DeclareMathOperator\csname #1\endcsname {#1}}
\DMO{ord}
\DMO{supp}
\DMO{Sym}
\DMO{cl}
\DMO{lcm}
\DMO{dist}
\DMO{tr}
\DMO{diam}
\DeclarePairedDelimiter\abs{\lvert}{\rvert}

\DeclarePairedDelimiter\norm{\lVert}{\rVert}

\DeclarePairedDelimiterX\spr[2]{\langle}{\rangle}{#1,#2}

\DeclarePairedDelimiterX\Set[2]{\{}{\}}{#1\colon #2}
\DeclarePairedDelimiterX\Seq[1]{(}{)}{#1}

\begin{document}
\title{Oscillation inequalities for Carleson--Dunkl operator}

\author{Wojciech S{\l}omian}
\address[Wojciech S{\l}omian]{
	  Instytut Matematyczny,
        Uniwersytet Wroc{\l}awski,
        Plac Grunwaldzki 2,
        50-384 Wroc{\l}aw,
        Poland}
    \email{wojciech.slomian@uwr.edu.pl}

\subjclass[2020]{42A38 (Primary), 42B25, 42B10}
\keywords{oscillation seminorm, Dunkl transform, radial functions, partial sums}   
\thanks{The author was
partially supported by the National Science Centre in Poland, grant Sonata Bis 2022/46/E/ST1/00036 and by the Foundation for Polish Science (START 068.2023)}

\begin{abstract}
In this paper, we establish estimates for the oscillation seminorm for the so-called Carleson--Dunkl operator on weighted $L^p(\RR,w(x)|x|^{2\alpha+1}{\rm d}x)$ spaces with power weights $w(x)=|x|^\beta$. As a result, we obtain oscillation estimates for the standard Carleson operator on $L_{\rm rad}^p(\RR^n,|x|^\beta{\rm d}x)$. As a byproduct, we obtain a transference principle for radial multipliers on $L_{\rm rad}^p$ spaces, in the spirit of the Rubio de Francia transference principle.

\end{abstract}

\maketitle

\section{Introduction and statement of results}

In this paper, we establish a uniform oscillation inequality on weighted $L^p$ spaces for the partial sums of the Dunkl transform. As a result, we are able to derive a uniform oscillation inequality for the standard Fourier transform on the real line and its higher-dimensional counterpart on radial functions. These results can be seen as a generalization of the oscillation inequality proven by Lacey and Terwilleger \cite{LT}. To establish our results, we utilize the recent observations made by Mirek, Szarek, and Wright \cite{MSW} regarding oscillation estimates of projection operators. This approach allows us to bypass time-frequency analysis and use the existing estimates instead. To apply these estimates, we employ techniques developed by Stempak and Trebels \cite{StempakTrebels1997} to transfer the estimates for the multipliers between Dunkl transforms of different orders. As a byproduct, we obtain a 'generalization' of the Rubio de Francia transference principle for radial multipliers from \cite{RdF2}.

\subsection{Statement of results}
Fix $\alpha \geq -1/2$ and let $f$ be a suitable function on $\mathbb{R}$. \textit{The Dunkl transform} of order $\alpha$ of function $f$ is defined as follows:
\begin{equation}\label{tfd}
\mathcal{D}_\alpha f(x): = \int_\mathbb{R} f(y) \frac{1}{2(yx)^\alpha} \big(J_{\alpha}(yx) - iJ_{\alpha+1}(yx)\big) |y|^{2\alpha+1} \, \mathrm{d}y, \quad x \in \mathbb{R},
\end{equation}
where $J_{\alpha}$ is the Bessel function of the first kind of order $\alpha$. The Dunkl transform was introduced by Dunkl \cite{Dunkl} as a generalization of the usual Fourier transform on $\mathbb{R}^n$ related to some group of reflections. In our case, we will be exploring the one-dimensional case. Simple calculations based on the formulas for $J_{1/2}$ and $J_{-1/2}$ show that $\mathcal{D}_{-1/2} = \mathcal{F}$, where
\begin{equation}\label{fouriertrans}
\mathcal{F} f(x) := \frac{1}{\sqrt{2\pi}} \int_\mathbb{R} f(y) e^{-ixy} \, \mathrm{d}y, \quad x \in \mathbb{R},
\end{equation}
is the Fourier transform. For more details about the Dunkl theory, we refer the reader to \cite{Ros1, Ros2} and the references therein.

For $1 \leq p < \infty$, $\alpha\geq-1/2$ and non-negative weight $w$ on $\mathbb{R}$, we denote by $L^p(\mathbb{R}, w|x|^{2\alpha+1}{\rm d}x)$ the space of all measurable functions $f$ on $\mathbb{R}$ for which
\begin{equation*}
    \|f\|_{L^p(\mathbb{R},w|x|^{2\alpha+1}\mathrm{d}x)}: = \left( \int_{-\infty}^\infty |f(x)|^p w(x)|x|^{2\alpha+1} \mathrm{d}x \right)^{\frac{1}{p}} < \infty.
\end{equation*}
The Dunkl transform $\mathcal{D}_\alpha$ extends uniquely to an isometric isomorphism on $L^2(\mathbb{R},|x|^{2\alpha+1}\mathrm{d}x)$ which satisfies the Plancherel identity
\begin{equation*}
    \|\mathcal{D}_\alpha f\|_{L^2(\mathbb{R},|x|^{2\alpha+1}\mathrm{d}x)} = \|f\|_{L^2(\mathbb{R},|x|^{2\alpha+1}{\rm d}x)}, \quad f \in L^2(\mathbb{R},|x|^{2\alpha+1}\mathrm{d}x).
\end{equation*}
Furthermore, the inverse Dunkl transform of order $\alpha$ is defined as
\begin{equation*}
    \widecheck{\mathcal{D}}_\alpha f(x) := \mathcal{D}_\alpha f(-x), \quad x \in \mathbb{R}.
\end{equation*}
Then for $f \in C_{\mathrm{c}}^\infty(\mathbb{R})$, we have $(\widecheck{\mathcal{D}}_\alpha \circ \mathcal{D}_\alpha) f = (\mathcal{D}_\alpha \circ \widecheck{\mathcal{D}}_\alpha)  f= f$. Simple proofs of these facts may be found in \cite[Proposition 1.3]{NowakStempak}.

In this paper, we will be interested in the partial sums of the Dunkl transform $\mathcal{D}_\alpha$. Namely, for any $f \in C_{\mathrm{c}}^\infty(\mathbb{R})$, we define the Dunkl transform partial sum operators
\begin{equation*}
    \mathcal{S}_{t}^\alpha f(x) := \widecheck{\mathcal{D}}_{\alpha}\big(\mathds{1}_{[-t,t]} \mathcal{D}_\alpha f\big)(x), \quad x \in \mathbb{R}.
\end{equation*}
\indent Let \( \mathbb{I} \) be any subset of \( \mathbb{R}_+ \).  
For a given increasing sequence \( I = (I_j : j \in \mathbb{N}) \subseteq \mathbb{I} \),  
the \emph{truncated oscillation seminorm} is defined for any family of measurable functions  
\( (\mathfrak{a}_t(x) : t > 0) \) with values in \( \mathbb{C} \), by setting
\begin{align}
\label{eq:45}
O_{I, J}^2(\mathfrak{a}_t(x): t \in \mathbb{I}) 
= \bigg( \sum_{j=1}^J \sup_{\substack{I_j \le t < I_{j+1} \\ t \in \mathbb{I}}} 
\left| \mathfrak{a}_t(x) - \mathfrak{a}_{I_j}(x) \right|^2 \bigg)^{1/2}
\quad \text{for all } J \in \mathbb{N} \cup \{\infty\}.
\end{align}

For \( J \in \mathbb{N} \cup \{\infty\} \), we denote by \( \mathfrak{S}_J(\mathbb{I}) \) the family of all strictly increasing sequences of length \( J + 1 \) contained in \( \mathbb{I} \). Furthermore, we denote by \( \mathbb{D} := \{ 2^n : n \in \mathbb{Z} \} \) the set of dyadic numbers.

Our main result reads as follows.
\begin{theorem}\label{main:oscil:dunkl}
Let $\alpha\geq-1/2$ and let $p\in[2,\infty)$. Assume that $\beta\in\RR$ is such that
\begin{equation*}
    -1<\beta+(\alpha+1/2)(2-p)<p/2-1.
\end{equation*}
When $p=2$, we additionally allow $\beta=0$. Then we have
\begin{equation*}
    \sup_{J \in \mathbb{N}} \sup_{I \in \mathfrak{S}_J(\mathbb{R}_+)} 
    \left\| O_{I, J}^2\big( \mathcal{S}_t^\alpha f : t > 0 \big) \right\|_{L^p(\mathbb{R},\, |x|^{\beta + 2\alpha + 1} \, \mathrm{d}x)} 
    \lesssim_{p, \alpha, \beta} 
    \| f \|_{L^p(\mathbb{R},\, |x|^{\beta + 2\alpha + 1} \, \mathrm{d}x)},
\end{equation*}
for any \( f \in L^p(\mathbb{R},\, |x|^{\beta + 2\alpha + 1} \, \mathrm{d}x) \).
 When $\beta=0$ then range specified above can be translated into
\begin{equation*}
    2\leq p<\frac{4(\alpha+1)}{2\alpha+1}.
\end{equation*}
In the special case \( \alpha = -\frac{1}{2} \), the right-hand side is understood to be \( +\infty \).

Furthermore, if $p\in(1,\infty)$ and $\beta\in\RR$ satisfies
\begin{equation}\label{main:oscil:dunkl:eq2}
    -1<\beta+(\alpha+1/2)(2-p)<p-1,
\end{equation}
then one has
\begin{equation*}
    \sup_{J\in\NN}\sup_{I\in\mathfrak{S}_J(\DD) }\big\|O_{I,J}^2\big(\calS_{2^n}^\alpha f:n\in\ZZ\big)\big\|_{L^p(\RR,|x|^{\beta+2\alpha+1}{\rm d}x)}\lesssim_{p,\alpha,\beta}\|f\|_{L^p(\RR,|x|^{\beta+2\alpha+1}{\rm d}x)},
\end{equation*}
for any $f\in L^p(\RR, |x|^{\beta+2\alpha+1}{\rm d}x)$. When $\beta=0$ the condition \eqref{main:oscil:dunkl:eq2} can be translated into
\begin{equation*}
    \frac{4(\alpha+1)}{2\alpha+3}< p<\frac{4(\alpha+1)}{2\alpha+1}.
\end{equation*}
In the special case \( \alpha = -\frac{1}{2} \), the right-hand side is understood to be \( +\infty \).
\end{theorem}
Since the one-dimensional Dunkl transform can be expressed in terms of the Hankel transform, and the latter is related to the Fourier transform of radial functions, we can use Theorem~\ref{main:oscil:dunkl} to derive an oscillation inequality for the partial sums of the Fourier transform over balls for radial functions.
\begin{corollary}
Let $n\geq1$ be a positive integer and let $p\in[2,\infty)$. Assume that $\beta\in\RR$ satisfies
\begin{equation*}
    -1<\beta+(n-1)(1-p/2)<p/2-1.
\end{equation*}
When $p=2$, we additionally allow $\beta=0$. Then we have
\begin{equation*}
    \sup_{J\in\NN}\sup_{I\in\mathfrak{S}_J(\RR_+) }\big\|O_{I,J}^2\big(\calF^{-1}(\mathds{1}_{B(0,t)}\calF f):t>0\big)\big\|_{L^p(\RR^n,|x|^{\beta}{\rm d}x)}\lesssim_{p,n,\beta}\|f\|_{L^p(\RR^n,|x|^{\beta}{\rm d}x)},
\end{equation*}
for any radial function $f\in L^p(\RR^n,|x|^{\beta}{\rm d}x)$. Furthermore, if $p\in(1,\infty)$ and $\beta\in\RR$ satisfy
\begin{equation*}
    -1<\beta+(n-1)(1-p/2)<p-1,
\end{equation*}
then we have
\begin{equation*}
    \sup_{J\in\NN}\sup_{I\in\mathfrak{S}_J(\DD) }\big\|O_{I,J}^2\big(\calF^{-1}(\mathds{1}_{B(0,2^m)}\calF f):m\in\ZZ\big)\big\|_{L^p(\RR^n,|x|^{\beta}{\rm d}x)}\lesssim_{p,n,\beta}\|f\|_{L^p(\RR^n,|x|^{\beta}{\rm d}x)},
\end{equation*}
for any radial function $f\in L^p(\RR^n,|x|^{\beta}{\rm d}x)$.
\end{corollary}
In the case when $n = 1$, we can strengthen our results, but first, let us recall the definition of the Muckenhoupt class $A_p$. For $p\in(1,\infty)$, the weight $w\colon\RR\to[0,\infty)$ is in the $A_p$ class if
\begin{equation*}
    \sup_{B}\Big(\frac{1}{|B|}\int_B w(x){\rm d}x\Big)\Big(\frac{1}{|B|}\int_Bw(x)^{-p'/p}{\rm d}x\Big)^{p/p'}<\infty,
\end{equation*}
where the supremum is taken over all intervals $B\subset\RR$ and $1/p+1/p'=1$. These classes, introduced by Muckenhoupt in \cite{Mu1}, serve to characterize the boundedness of the Hardy--Littlewood maximal functions on weighted Lebesgue spaces $L^p(\RR,w{\rm d}x)$.
\begin{proposition}\label{prop:weight:carle}
Let $p\in(2,\infty)$ and $w\in A_{p/2}$. Then we have
\begin{equation*}
    \sup_{J\in\NN}\sup_{I\in\mathfrak{S}_J(\RR_+) }\big\|O_{I,J}^2\big(\calF^{-1}(\mathds{1}_{[-t,t]}\calF f):t>0\big)\big\|_{L^p(\RR,w{\rm d}x)}\lesssim_{p,w}\|f\|_{L^p(\RR,w{\rm d}x)},
\end{equation*}
for any $f\in L^p(\RR,w{\rm d}x) $. In the case of the power weight $w(x)=|x|^\beta$ we additionally allow $p=2$. Furthermore if, $p\in(1,\infty)$ and $w\in A_p$, then we have
\begin{equation*}
    \sup_{J\in\NN}\sup_{I\in\mathfrak{S}_J(\DD) }\big\|O_{I,J}^2\big(\calF^{-1}(\mathds{1}_{[-2^n,2^n]}\calF f):n\in\ZZ\big)\big\|_{L^p(\RR,w{\rm d}x)}\lesssim_{p,w}\|f\|_{L^p(\RR,w{\rm d}x)},
\end{equation*}
for any $f\in L^p(\RR,w{\rm d}x)$.
\end{proposition}
The above result is a generalization of the result proven by Lacey and Terwilleger \cite{LT} to the setting of weighted spaces with weights from $A_p$. It is believed that the oscillation seminorm is some kind of endpoint of the $r$-variation seminorm corresponding to $r=2$, although currently, we do not know the exact nature of that endpoint. Recall that for any $r \geq 1$, the $r$-variation seminorm is defined for any family of measurable functions $(\mathfrak{a}_t(x): t > 0)$ by setting
\begin{align*}
    V^{r}(\mathfrak{a}_t(x) : t > 0) = \sup_{J \in \mathbb{N}} \sup_{\substack{t_{0} < \cdots < t_{J} \\ t_{j} > 0}} \left( \sum_{j = 0}^{J-1} \big|\mathfrak{a}_{t_{j+1}}(x) - \mathfrak{a}_{t_j}(x)\big|^{r} \right)^{1/r}.
\end{align*}
For more details, we refer to \cite{MSS}. In this sense, our result can be seen as a closure of the variational estimates obtained by Di Plinio, Do, and Uraltsev \cite[Theorem 2]{DDU}, who established that for any $r \in (2, \infty)$, $p > r'$, where $1/r + 1/r' = 1$, and any $w \in A_{p/r'}$, one has
\begin{equation*}
    \big\|V^r(\mathcal{F}^{-1}(\mathds{1}_{[-t,t]}\mathcal{F} f) : t > 0)\big\|_{L^p(\mathbb{R}, w \, \mathrm{d}x)} \lesssim_{p, w} \|f\|_{L^p(\mathbb{R}, w \, \mathrm{d}x)}.
\end{equation*}

\begin{theorem}[Rubio de Francia transference theorem for radial multipliers]\label{thm:RdFR}
Let $p\in(1,\infty)$ and let $-1<\beta<p-1$ be such that the sequence of even functions $(m_k)_{k\in\ZZ}$ is a vector-valued bounded multiplier on  $L^p(\RR,|x|^\beta {\rm d}x)$, that is
\begin{equation}\label{thm:RdFR:eq1}
    \Big\|\Big(\sum_{k\in\ZZ}\big|\calF^{-1}(m_k\calF f)\big|^2\Big)^{1/2}\Big\|_{L^p(\RR,|x|^\beta{\rm d}x)}\lesssim\|f\|_{L^p(\RR,|x|^\beta {\rm d}x)}.
\end{equation}
Then for any $n\geq 1$ we have
\begin{equation}\label{thm:RdFR:eq2}
    \Big\|\Big(\sum_{k\in\ZZ}\big|\calF^{-1}(m_k(|\cdot|)\calF f)\big|^2\Big)^{1/2}\Big\|_{L^p(\RR^n,|x|^{\beta^*}{\rm d}x)}\lesssim_{n}\|f\|_{L^p(\RR^n,|x|^{\beta^*} {\rm d}x)},
\end{equation}
for any radial function $f\in L_{\rm rad}^p(\RR^n,|x|^\beta{\rm d}x)$, where $\beta^*:=\beta-(n-1)(1-p/2)$ and $|x|$ is the standard Euclidean norm of $x\in\RR^n$.
\end{theorem}
The classical Rubio de Francia transference theorem \cite[Theorem 2.1]{RdF2} states that for $p=2$ and $-1<\beta<1$ the inequality \eqref{thm:RdFR:eq1} implies that \eqref{thm:RdFR:eq2} holds for any $f\in L^2(\RR^n,|x|^{\beta} {\rm d}x)$. Our result is a generalization of this result to $p\neq 2$. In such a general setting, we can consider the subset of radial functions only, as shown by the famous disc multiplier theorem proven by Fefferman \cite{Fefferman1971}. It is worth noting that our proof relies heavily on the nature of radial functions and does not cover the original Rubio de Francia result.\\
\textbf{Acknowledgments.}
    The author would like to express his gratitude to the anonymous referees for their precise comments and remarks, which have significantly improved this paper.

\subsection{Notation}\label{sec:notation}
We denote $\NN:=\{1, 2, \ldots\}$ and $\RR_+:=(0, \infty)$. For $n\in\NN$ the sets $\ZZ^n$, $\RR^n$ have the usual meaning. 

For a function $f$ on $\RR$ let us denote by $f_{\rm e}$ and $f_{\rm o}$ its even and odd parts restricted to $\RR_+$ , i.e.
\begin{equation*}
	    f_{\rm e}(x):=\frac12(f(x)+f(-x)), \quad f_{\rm o}(x):=\frac12(f(x)-f(-x)), \quad x\in\RR_+.
\end{equation*}

By $\calF$ we denote the Fourier transform defined by \eqref{fouriertrans} and by $\calF^{-1}$ its inverse transformation.

By \( C_{\rm c}^\infty(X) \), where \( X = \mathbb{R} \) or \( X = \mathbb{R}_+ \), we denote the set of all smooth functions with compact support in \( X \).

We write $A \lesssim B$ to indicate that $A\le CB$ with a constant $C>0$. The constant $C$ may vary from line to line. We write $\lesssim_{\delta}$ if the implicit constant depends on $\delta$. 
\section{Preliminaries}
\subsection{Dunkl and Hankel transforms}
The Dunkl transform on the real line can be seen as a symmetrization of the Hankel transform. First, recall that for given $\alpha\geq-1/2$ \textit{the Hankel transform} of order $\alpha$ of a suitable function $f$ on $\RR_+$ is defined by
\begin{equation*}
	\calH_\alpha f(x):=\int_0^\infty{f(y)\frac{J_\alpha(xy)}{(xy)^\alpha}y^{2\alpha+1}{\rm d}y}, \quad x\in\RR_+.
\end{equation*}
Then the Dunkl transform $\calD_\alpha$ can be expressed as
 \begin{equation}\label{DunkHankel}
     \calD_\alpha f(x)=\calH_\alpha f_{\rm e}(|x|)-ix\calH_{\alpha+1}\big(f_{\rm o}/(\cdot)\big)(|x|),\quad x\in\RR.
 \end{equation}
In analogy to the Dunkl transform, for any $f\in C_{\rm c}^\infty(\RR_+)$, we define the partial sums of the Hankel transform as
\begin{equation*}
    \widetilde{\calS}_{t}^\alpha f(x):={\calH}_{\alpha}\big(\mathds{1}_{[0,t]} {\calH}_\alpha f\big)(x),\quad x\in\RR_+.
\end{equation*}
Having defined the partial sum operators, we may define the corresponding maximal functions. For $f\in C_{\mathrm{c}}^\infty(\mathbb{R})$, we define the \textit{Carleson--Dunkl} operator as
 \begin{equation}
     \calC_{\ast}^\alpha f(x):=\sup_{t>0} |\calS_{t}^\alpha f|(x),\quad x\in\RR.
 \end{equation}
Similarly, we define the \textit{Carleson--Hankel} operator as
\begin{equation}
    \widetilde{\calC}_{\ast}^\alpha f(x):=\sup_{t>0} |\widetilde{\calS}_{t}^\alpha f|(x),\quad x\in\RR_+.
\end{equation}
The $L^p$ boundedness of the Carleson--Hankel operator was established by Kanjin \cite{Kanjin88} and independently by Prestini \cite{Prestini88} at the end of the 80s.

\begin{theorem}[Kanjin--Prestini]\label{thmKP}
Let $\alpha\geq-1/2$ and let $p\in(1,\infty)$ satisfy
 \begin{equation}
     \frac{4(\alpha+1)}{2\alpha+3}<p<\frac{4(\alpha+1)}{2\alpha+1}.
 \end{equation}
There exists a constant $C_{p,\alpha}>0$ such that for any $f\in L^p(\mathbb{R}_+,x^{2\alpha+1}{\rm d}x)$, we have
\begin{equation}\label{eq:1:CHest}
    \big\|\widetilde{\calC}_{\ast}^\alpha f\big\|_{L^p(\RR_+,x^{2\alpha+1}{\rm d}x)}\leq C_{p,\alpha}\|f\|_{L^p(\RR_+,x^{2\alpha+1}{\rm d}x)}.
\end{equation}
\end{theorem}
Of particular interest to us will be the approach discovered by Prestini \cite{Prestini88}. Namely, in order to prove \eqref{eq:1:CHest}, she provided a pointwise estimate of the Hankel transform partial sums $\widetilde{\calS}_{t}^\alpha f(x)$ in terms of the Hardy--Littlewood maximal function, the maximal Hilbert transform, the conjugate of the Hardy operator, and the classical Carleson--Hunt operator.
\begin{theorem}[Prestini]\label{thm:Prestini}
Let $\alpha\geq-1/2$ and let $f\in C_{\rm c}^\infty(\RR_+)$. There exists a constant $C_\alpha>0$ such that
\begin{equation*}
\big|\calH_{\alpha}\big(\mathds{1}_{[0,t]}\calH_\alpha f\big)(|x|)\big|\leq C_\alpha |x|^{-(\alpha+1/2)}(M_{\rm HL}+H+\calH_\ast+\calC)((\cdot)^{\alpha+1/2}f)(|x|),\quad x\in\RR,
 \end{equation*}
 where
\begin{equation}\label{eq:2:HL}
    M_{\rm HL}f(x):=\sup_{r>0}\frac{1}{2r}\int_{-r}^r |f(x-y)|{\rm d}y,\quad x\in\RR,
\end{equation}
is the Hardy--Littlewood maximal function,
\begin{equation}\label{eq:3:H}
    Hf(x):=\int_{|x|}^\infty\frac{|f(y)|}{y}{\rm d}y,\quad x\in\RR,
\end{equation}
is the conjugate Hardy operator,
\begin{equation}\label{eq:4:Hilb}
    \calH_\ast f(x):=\sup_{\varepsilon>0}\Bigg|\int_{|y|>\varepsilon}\frac{f(x-y)}{y}{\rm d}y\Bigg|,\quad x\in\RR,
\end{equation}
is the maximal Hilbert transform, and
\begin{equation}\label{eq:5:CH}
    \calC f(x):=\sup_{\varepsilon>0}\sup_{\xi\in\RR}\Bigg|\int_{|y|>\varepsilon}{e^{i\xi y}}\frac{f(x-y)}{y}{\rm d}y\Bigg|,\quad x\in\RR
\end{equation}
is the Carleson--Hunt maximal operator. The function \( f \) on the right-hand side of the statement is understood as a function on \( \mathbb{R} \), defined by setting \( f(x) \) for \( x \in \mathbb{R}_+ \), and \( 0 \) otherwise.

\end{theorem}
Several years following Prestini's work, Romera and Soria \cite{RomeraSoria91} provided an alternative proof of the aforementioned theorem. Prestini's estimate plays a pivotal role in our paper, enabling us to work with well-established operators \eqref{eq:2:HL}--\eqref{eq:5:CH} instead of $\calC_{\ast}^\alpha$ and $\widetilde{\calC}_{\ast}^\alpha$.

In 2007, El Kamel and Yacoub \cite{EKY} utilized the formula \eqref{DunkHankel} together with Prestini's estimate and proved the $L^p$ boundedness of the Carleson--Dunkl operator $\calC_{\ast}^\alpha$.
\begin{theorem}[El Kamel--Yacoub]
Let $\alpha\geq-1/2$ and let $p\in(1,\infty)$ satisfy
 \begin{equation}
     \frac{4(\alpha+1)}{2\alpha+3}<p<\frac{4(\alpha+1)}{2\alpha+1}.
 \end{equation}
There exists a constant $C_{p,\alpha}>0$ such that for any $f\in L^p(\mathbb{R},|x|^{2\alpha+1}{\rm d}x)$, we have
\begin{equation}
    \big\|\calC_{\ast}^\alpha f\big\|_{L^p(\RR,|x|^{2\alpha+1}{\rm d}x)}\leq C_{p,\alpha}\|f\|_{L^p(\RR,|x|^{2\alpha+1}{\rm d}x)}.
\end{equation}
\end{theorem}

\subsection{Weighted estimates with weights from Muckenhoupt class $A_p$}
The next theorem is well-known, but we will provide its ``proof'' to give relevant references.
\begin{theorem}\label{thm:weights}
Let $p\in(1,\infty)$. Let $w\colon\RR\to[0,\infty)$ be a weight from the $A_p$ class. Let $T$ be one of the operators \eqref{eq:2:HL}--\eqref{eq:5:CH}. Then there exists a constant $C_{w,p}:=C_{w,p}(T)>0$ such that the weighted estimate
\begin{equation}\label{eq:6:weight}
    \big\|Tf\|_{L^p(\RR,w{\rm d}x)}\leq C_{w,p}\|f\|_{L^p(\RR,w{\rm d}x)},\quad f\in L^p(\RR,w{\rm d}x),
\end{equation}
holds.
\end{theorem}
\begin{proof}
The weighted inequality for the maximal function $M_{\rm HL}$ is the primary reason for considering $A_p$ classes, and it was proven by Muckenhoupt \cite{Mu1}. The inequality \eqref{eq:6:weight} for the conjugate Hardy operator $H$ follows by duality and by the fact that $w\in A_p$ if and only if $w^{1-p'}\in A^{p'}$ where $1/p+1/p'=1$. Namely, one has
\begin{align*}
    \|Hf\|_{L^p(\RR,w{\rm d}x)}&=\sup_{g:\, \|g\|_{L^{p'}(\RR,w^{1-p'}{\rm d}x)}\leq1}\bigg|\int_{\RR}Hf(x)g(x){\rm d}x\bigg|\\
    &=\sup_{g:\, \|g\|_{L^{p'}(\RR,w^{1-p'}{\rm d}x)}\leq1}\int_{\RR}|f(y)|\frac{1}{y}\int_{|x|<y}|g(x)|{\rm d}x{\rm d}y
    \lesssim\|f\|_{L^p(\RR,w{\rm d}x)}\|M_{\rm HL}g\|_{L^{p'}(\RR,w^{1-p'}{\rm d}x)}.
\end{align*}
For the maximal Hilbert transform $\calH_\ast$, it follows from the general theory of weighted inequalities for singular integrals, see \cite[Chapter 9]{Graf2}. Finally, the weighted estimates for the Carleson operator $\calC$ were proven by Hunt and Young \cite{HuntYoung74}.
\end{proof}

Now, let $w\colon\RR\to[0,\infty)$ be a weight such that
\begin{equation}\label{eq:7:Apa}
    w(x)|x|^{2\alpha+1-p(\alpha+1/2)}\in A_p.
\end{equation}
The set of such weights will be called $A_p^\alpha$. Obviously, $A_p^{-1/2}=A_p$. The following result is a straightforward generalization of Theorem~\ref{thmKP}.
\begin{proposition}[Weighted Kanjin--Prestini estimate]\label{prop:romera}
Let $\alpha\geq-1/2$ and let $p\in(1,\infty)$. Let $w\colon\RR\to[0,\infty)$ be an even weight from $A_p^\alpha$. There exists a constant $C_{p,\alpha,w}>0$ such that for any $f\in L^p(\mathbb{R}_+,wx^{2\alpha+1}{\rm d}x)$, we have
\begin{equation}\label{wKPe}
    \big\|\widetilde{\calC}_{\ast}^\alpha f\big\|_{L^p(\RR_+,wx^{2\alpha+1}{\rm d}x)}\leq C_{p,\alpha,w}\|f\|_{L^p(\RR_+,wx^{2\alpha+1}{\rm d}x)}.
\end{equation}
\end{proposition}
\begin{proof}
This result follows from Theorems~\ref{thm:Prestini} and \ref{thm:weights}.
\end{proof}
This result was first proved by Romera \cite{Romera} in the mid-90s and generalizes the famous result of Herz \cite{Herz} (unweighted setting) and Andersen \cite{Andersen} (weighted setting). The result states that the Hankel transform partial sum operator $\widetilde{\calS}_{t}^\alpha$ is uniformly bounded on $L^p(\RR_+,wx^{2\alpha+1},dx)$ with $w\in A_p^\alpha$.

The results proven by El Kamel, Yacoub, and Romera suggests that the stated result holds true.
\begin{proposition}[Weighted Carleson--Dunkl estimate]\label{prop:WCD}
Let $\alpha\geq-1/2$ and let $p\in(1,\infty)$. Let $w\colon\RR\to[0,\infty)$ be an even weight from $A_p^\alpha$. There exists a constant $C_{p,\alpha,w}>0$ such that for any $f\in L^p(\mathbb{R},w|x|^{2\alpha+1}{\rm d}x)$, we have
\begin{equation}\label{eq8:wCD}
    \big\|{\calC}_{\ast}^\alpha f\big\|_{L^p(\RR,w|x|^{2\alpha+1}{\rm d}x)}\leq C_{p,\alpha,w}\|f\|_{L^p(\RR,w|x|^{2\alpha+1}{\rm d}x)}.
\end{equation}
\end{proposition}
\begin{proof}
The proof is a simple repetition of the arguments given in \cite{EKY}. Below, we provide a sketch of the proof. Let $w\in A_p^\alpha$ be an even weight, and let $f\in C_{\rm c}^\infty(\RR)$. By \cite[Lemma 2.1]{EKY}, we see that
\begin{equation*}
    {\calS}_{t}^\alpha f(x)=\widetilde{\calS}_{t}^\alpha f_{\rm e}(|x|)+x\widetilde{\calS}_{t}^{\alpha+1}\big(f_{\rm o}/{(\cdot)}\big)(|x|),\quad x\in\RR,
\end{equation*}
where $f_{\rm o}/{(\cdot)}$ stands for the function $x \mapsto f_{\rm o}(x) / x$. Therefore, one can write
\begin{equation}\label{eq:decomposition}
    {\calC}_{\ast}^\alpha f(x)\leq \widetilde{\calC}_{\ast}^\alpha(f_{\rm e})(|x|)+|x|\widetilde{\calC}_{\ast}^{\alpha+1}\big(f_{\rm o}/{|\cdot|}\big)(|x|),\quad x\in\RR.
\end{equation}
Now, it is easy to see that
\begin{equation*}
    \big\|\widetilde{\calC}_{\ast}^\alpha(f_{\rm e})(|x|)\big\|_{L^p(\RR,w|x|^{2\alpha+1}{\rm d}x)}=2^{1/p}\big\|\widetilde{\calC}_{\ast}^\alpha(f_{\rm e})\big\|_{L^p(\RR_+,wx^{2\alpha+1}{\rm d}x)}
\end{equation*}
and by Proposition~\ref{prop:romera} we get that
\begin{equation*}
    \big\|\widetilde{\calC}_{\ast}^\alpha(f_{\rm e})(|x|)\big\|_{L^p(\RR,w|x|^{2\alpha+1}{\rm d}x)}\lesssim_\alpha\|f_{\rm e}\|_{L^p(\RR_+,wx^{2\alpha+1}{\rm d}x)}=\|f\|_{L^p(\RR,w|x|^{2\alpha+1}{\rm d}x)}.
\end{equation*}
Now, by using Theorem~\ref{thm:Prestini} with $\alpha+1$ we get

\begin{equation*}
    \Big\||x|\widetilde{\calC}_{\ast}^{\alpha+1}\big(f_{\rm o}/{|\cdot|}\big)(|x|)\Big\|_{L^p(\RR,w|x|^{2\alpha+1}{\rm d}x)}\lesssim_\alpha\big\|K(f_{\rm o}(\cdot)^{\alpha+1/2})\|_{L^p(\RR,w|x|^{2\alpha+1-p(\alpha+1/2)}{\rm d}x)},
\end{equation*}
where $K:=M_{\rm HL}+H+\calH_\ast+\calC$. Now, since the condition \eqref{eq:7:Apa} is satisfied, we see that \eqref{eq8:wCD} holds.
\end{proof}
The above result seems to be new despite the fact that its proof is a trivial adaptation of well-known arguments.
\begin{corollary}
Let $\alpha\geq-1/2$, $p\in(1,\infty)$, and let $w$ be an even weight from $A_p^\alpha$. Then, for any $f\in L^p(\mathbb{R},w|x|^{2\alpha+1}{\rm d}x)$, we have
\begin{equation*}
    S_t^\alpha f(x) \to f(x) \quad \text{as } t \to \infty, \quad \text{for a.e. } x \in \mathbb{R}.
\end{equation*}

In particular, when $w(x)=|x|^\beta$, we get that the partial sum of the Dunkl transform of $f\in L^p(\RR,|x|^{\beta+2\alpha+1}{\rm d}x)$ converges almost everywhere when
\begin{equation*}
    -1<\beta+(\alpha+1/2)(2-p)<p-1.
\end{equation*}
\end{corollary}
If we take $\alpha=-1/2$ in Proposition~\ref{prop:WCD}, we get that for $w\in A_p$ the Carleson operator is bounded on $L^p(\RR,wdx)$ to itself. For $p=2$, we get that the Carleson--Dunkl ${\calC}_{\ast}^\alpha$ operator is bounded on $L^2(\RR,w|x|^{2\alpha+1}{\rm d}x)$ if $w$ is an even weight from $A_2$. One may ask the question of whether the condition $w\in A_p^\alpha$ is optimal in any sense. This does not seem to be the case even when $p=2$. Let us consider $w_{a,b}(x)=|x|^a(1+|x|)^{b-a}$. Then $w_{a,b}\in A_2$ if and only if $a,b\in(-1,1)$. On the other hand Betancor, Ciaurri and Varona \cite{BCV} proved that 
\begin{equation*}
    \big\|\widecheck{\calD}_{\alpha}\big(\mathds{1}_{[-1,1]} {\calD}_\alpha f\big)\big\|_{L^2(\RR,w_{a,b}|x|^{2\alpha+1}){\rm d}x}\lesssim\|f\|_{L^2(\RR,w_{a,b}|x|^{2\alpha+1}{\rm d}x)}
\end{equation*}
if and only if
\begin{equation*}
    -(2\alpha+2)<a<2\alpha+2\quad \text{and}\quad -1<b<1.
\end{equation*}
This suggests that $w\in A_p^\alpha$ is not optimal even in the case $p=2$. 

A natural candidate for the condition for which the weighted inequality \eqref{eq8:wCD} holds is the Muckenhoupt condition with an appropriate measure. Namely, we conjecture the following.
\begin{conj}
Let $w\colon\RR\to[0,\infty)$ satisfy the following condition
\begin{equation*}
    \sup_{B}\Big(\frac{1}{|B|}\int_B w(x)|x|^{2\alpha+1}{\rm d}x\Big)\Big(\frac{1}{|B|}\int_Bw(x)^{-p'/p}|x|^{2\alpha+1}{\rm d}x\Big)^{p/p'}<\infty,
\end{equation*}
where the supremum is taken over all intervals $B\subset\RR$ and $1/p+1/p'=1$. Then the inequality
\begin{equation*}
    \big\|{\calC}_{\ast}^\alpha f\big\|_{L^p(\RR,w|x|^{2\alpha+1}{\rm d}x)}\leq C_{p,\alpha,w}\|f\|_{L^p(\RR,w|x|^{2\alpha+1}{\rm d}x)}
\end{equation*}
holds for any $f\in L^p(\mathbb{R},w|x|^{2\alpha+1}{\rm d}x)$.
\end{conj}

\subsection{Modified Dunkl and Hankel transform}
Recall that for given $\alpha\geq-1/2$ \textit{the modified Hankel transform} of order $\alpha$ of a suitable function $f$ on $\RR_+$ is defined by
\begin{equation}\label{modHankel}
	H_\alpha f(x)=\int_0^\infty{f(y)(xy)^{1/2}J_\alpha(xy){\rm d}x}, \quad x\in\RR_+.
\end{equation}
Motivated by the formula \eqref{DunkHankel} Nowak and Stempak \cite{NowakStempak} have introduced \textit{the modified Dunkl transform} as
\begin{equation}\label{modedunkl}
	    D_\alpha f(x):=H_\alpha(f_{\rm e})(|x|)-i\sgn(x) H_{\alpha+1}(f_{\rm o})(|x|),\quad x\in\RR.
\end{equation}
The modified Dunkl transform exhibits similar properties to the standard one. Namely, the operator $D_\alpha$ extends uniquely to an isometric isomorphism on $L^2(\mathbb{R}, \mathrm{d}x)$, which satisfies the Plancherel identity
\begin{equation*}
    \|D_\alpha f\|_{L^2(\mathbb{R}, \mathrm{d}x)} = \|f\|_{L^2(\mathbb{R}, \mathrm{d}x)}, \quad f \in L^2(\mathbb{R}, \mathrm{d}x).
\end{equation*}
Moreover, its inverse is given by
\begin{equation*}
    \widecheck{D}_\alpha f(x) := D_\alpha f(-x), \quad x \in \mathbb{R}.
\end{equation*}
Then for $f \in C_{\mathrm{c}}^\infty(\mathbb{R})$, we have $(\widecheck{D}_\alpha \circ D_\alpha) f = (D_\alpha \circ \widecheck{D}_\alpha) f = f$. Proofs of these facts may be found in \cite[Proposition 1.3]{NowakStempak}.

The standard Dunkl transform and its modified counterpart are related to each other by the following relation:
\begin{equation}\label{dunksrelation}
    \mathcal{D}_\alpha = M_{-(\alpha+1/2)} \circ D_\alpha \circ M_{\alpha+1/2},
\end{equation}
where for any $a \in \mathbb{R}$, the operator $M_a f(x) := |x|^a f(x)$ is the multiplication operator.

The key feature of the modified Dunkl transform is the fact that for any $\alpha \geq -\frac{1}{2}$, the operator $D_\alpha$ acts on the same space, in contrast to the standard Dunkl transform, which requires the measure $|x|^{2\alpha+1} \, \mathrm{d}x$.
%%%%%%%%%%%%%%%%%%%%%%%%%%%%%%%%%%%%%%%%%%%%%%%%%%%%%%%%%%%%%%%%%%%%%%%%%%%%%%%%%%%%%%%%%%%%%%%%
\subsection{Oscillation seminorms for projection operators}
In the case of seminorms, we will follow the notation from \cite{MSS} and \cite{MSW}. Let $(P_t)_{t \in \II}$ be a family of projections; that is, linear operators satisfying
\begin{align}\label{eq:10}
P_s P_t = P_{\min\{s,t\}}.
\end{align}
It is easy to see that the partial sum operators $\calS_t^\alpha$ are projection operators in the above sense. The next result, proven by Mirek, Szarek, and Wright, shows that the oscillation estimates can be deduced from certain vector-valued estimates.

\begin{theorem}{\cite[Theorem 3.1]{MSW}} \label{thm:p}
Let  $\II\subseteq \RR$ be such that $\#\II\ge2$. Let $(P_t)_{t \in \II}$ be a family of projections. If the set $\II$ is uncountable then we assume in addition that $\II\ni t\mapsto P_tf$ is
continuous $\mu$-almost everywhere on $X$. Let $p\in(1, \infty)$ be fixed. Suppose that
$P_t$ are bounded on $L^p(\RR,{\rm d}\mu)$, and suppose that the
following two estimates hold
\begin{align} \label{1.15}
\sup_{J\in\ZZ_+}\sup_{I\in \mathfrak{S}_J(\II)} 
\norm[\Big]{ \Big( \sum_{j=0}^{J-1} 
\abs{ (P_{I_{j+1}}- P_{I_{j}}) f}^2 \Big)^{1/2} }_{L^p(\RR,{\rm d}\mu)}
\lesssim_{p} 
\norm{f}_{L^p(\RR,{\rm d}\mu)}, \qquad f \in L^p(\RR,{\rm d}\mu),
\end{align}
and the vector-valued estimate
\begin{align} \label{1.2}
\norm[\Big]{ \Big( \sum_{j \in \ZZ} \sup_{t \in \II} \abs{P_{t} f_j}^2 \Big)^{1/2} }_{L^p(\RR,{\rm d}\mu)} 
\lesssim_{p} 
\norm[\Big]{ \Big( \sum_{j \in \ZZ} \abs{f_j}^2 \Big)^{1/2} }_{L^p(\RR,{\rm d}\mu)}, \qquad (f_j)_{j\in\ZZ}\in L^p(\RR,{\rm d}\mu;\ell^2(\ZZ)).
\end{align}
Then the following oscillation estimate holds: 
\begin{align} \label{1.3}
\sup_{J \in \ZZ_+} \sup_{ I \in \mathfrak S_{J}(\II) }
\norm{ O^2_{I, J} ( P_{t} f : t \in \II) }_{L^p(\RR,{\rm d}\mu)} 
\lesssim_{p}
\norm{ f }_{L^p(\RR,{\rm d}\mu)}, 
\qquad f \in L^p(\RR,{\rm d}\mu).
\end{align}
\end{theorem}
Now we are ready to give the proof of Proposition~\ref{prop:weight:carle}
\begin{proof}[Proof of Proposition~\ref{prop:weight:carle}]
By Theorem~\ref{thm:p}, to establish the first part, we need to show two estimates:
For any \( w \in A_p \), we have the vector-valued estimate for the Carleson operator,
\begin{equation}\label{eq:prop1.8proof}
    \Big\|\Big(\sum_{n\in\ZZ}|\calC f_n|^2\Big)^{1/2}\Big\|_{L^p(\RR,w\,{\rm d}x)} \leq C_{p,w} \Big\|\Big(\sum_{n\in\ZZ}|f_n|^2\Big)^{1/2}\Big\|_{L^p(\RR,w\,{\rm d}x)}
\end{equation}
and for any \( p \in (2,\infty) \) and any \( w \in A_{p/2} \) we have the Rubio de Francia estimate
\begin{equation}\label{eq:prop1.8proof2}
\Big\|\Big(\sum_{j\in\ZZ}\big|\calF^{-1}(\mathds{1}_{I_j} \calF f)\big|^2\Big)^{1/2}\Big\|_{L^p(\RR,\,w{\rm d}x)} \lesssim \|f\|_{L^p(\RR,\,w{\rm d}x)}, \quad f \in L^p(\RR,\,w{\rm d}x),
\end{equation}
where $(I_j:j\in\ZZ)\subset\RR$ is a family of pairwise disjoint intervals.

The vector-valued estimate for the Carleson operator was proved by Grafakos, Martell, and Soria \cite{GMS}. On the other hand, the inequality \eqref{eq:prop1.8proof2} was proved by Rubio de Francia \cite{RdF}. Moreover, when \( w(x) = |x|^\beta \), Rubio de Francia \cite[Theorem 4.2]{RdF} showed that \eqref{eq:prop1.8proof2} holds for \( p = 2 \) and \( -1 < \beta \leq 0 \).

For the second part of the theorem, we need to prove \eqref{eq:prop1.8proof2} but with the dyadic intervals \( I_j := (2^{k_j}, 2^{k_{j+1}}) \), \( k_j, k_{j+1} \in \ZZ \) and \( w \in A_p \). This follows from the weighted Littlewood--Paley theory established by Kurtz \cite{Kurtz}.
\end{proof}

In a similar manner, in order to prove Theorem~\ref{main:oscil:dunkl} we need to show two estimates:
\begin{equation}\label{eq9:vector-va}
    \Big\|\Big(\sum_{n\in\ZZ}\sup_{t>0} |\calS_{t}^\alpha f_n|^2\Big)^{1/2}\Big\|_{L^p(\RR,w|x|^{2\alpha+1}{\rm d}x)}\lesssim_{\alpha,w}\Big\|\Big(\sum_{n\in\ZZ}|f_n|^2\Big)^{1/2}\Big\|_{L^p(\RR,w|x|^{2\alpha+1}{\rm d}x)},
\end{equation}
where $w\in A_p^\alpha$ is an even weight and $(f_n)_{n\in\ZZ}\in L^p(\RR, w|x|^{2\alpha+1}{\rm d}x;\ell^2(\ZZ))$ and
\begin{equation}\label{RdF eq:1}
\Big\|\Big(\sum_{j\in\ZZ}\big|\widecheck{\calD}_\alpha(\mathds{1}_{I_j} \calD_\alpha f)\big|^2\Big)^{1/2}\Big\|_{L^p(\RR,|x|^{\beta+2\alpha+1}{\rm d}x)}\lesssim\|f\|_{L^p(\RR,|x|^{\beta+2\alpha+1}{\rm d}x)}.
\end{equation}
where $(I_j:j\in\ZZ)\subset\RR$ is a family of pairwise disjoint intervals and 
\begin{equation*}
    -1<\beta+(\alpha+1/2)(2-p)<p/2-1.
\end{equation*}
For the second part of Theorem~\ref{main:oscil:dunkl} we need to show \eqref{RdF eq:1} but with disjoint dyadic intervals $I_j:=(2^{k_j},2^{k_{j+1}})$, $k_j,k_{j+1}\in\ZZ$ and
\begin{equation*}
    -1<\beta+(\alpha+1/2)(2-p)<p-1.
\end{equation*}
We will establish \eqref{eq9:vector-va} and \eqref{RdF eq:1} in Section 3.

\section{Oscillation inequality for partial sums of Dunkl transform}

\subsection{Vector-valued estimate \eqref{eq9:vector-va}}
This estimate is a consequence of weighted vector-valued estimates for operators \eqref{eq:2:HL}--\eqref{eq:5:CH}. Namely, by \eqref{eq:decomposition} we see that
\begin{equation*}
   |\calC_{\ast}^\alpha f_n|(x)\leq \widetilde{\calC}_{\ast}^\alpha(f_{n,{\rm e}})(|x|)+|x|\widetilde{\calC}_{\ast}^{\alpha+1}\big(f_{n,{\rm o}}/{|\cdot|}\big)(|x|),
\end{equation*}
where $f_{n,{\rm e}}$ and $f_{n,{\rm o}}$ are the even and odd part of $f_n$, respectively. By Theorem~\ref{thm:Prestini} we see that
\begin{equation*}
    |\calC_{\ast}^\alpha f_n|(x)\lesssim_\alpha|x|^{-(\alpha+1/2)}\big(K((\cdot)^{\alpha+1/2})f_{n,{\rm e}})+K((\cdot)^{\alpha+1/2})f_{n,{\rm o}})\big)(|x|),
\end{equation*}
where $K:=M_{\rm HL}+H+\calH_\ast+\calC$. Consequently, in order to show \eqref{eq9:vector-va} it is enough to establish
\begin{equation}\label{eq10:vector-va}
\Big\|\Big(\sum_{n\in\ZZ}|Kf_n|^2\Big)^{1/2}\Big\|_{L^p(\RR,w|x|^{2\alpha+1-p(\alpha+1/2)}{\rm d}x)}\lesssim_{\alpha,w}\Big\|\Big(\sum_{n\in\ZZ}|f_n|^2\Big)^{1/2}\Big\|_{L^p(\RR,w|x|^{2\alpha+1-p(\alpha+1/2)}{\rm d}x)},
\end{equation}
for any sequence of functions $(f_n)_{n\in\ZZ}\in{L^p(\RR,w|x|^{2\alpha+1-p(\alpha+1/2)}{\rm d}x;\ell^2(\ZZ))}$. Since the operator $K$ is the sum of four operators, it is sufficient to show an appropriate estimate for each component. Let $w\in A_p$. The weighted vector-valued estimate for the Hardy--Littlewood maximal function $M_{\rm HL}$, 
\begin{equation*}
    \Big\|\Big(\sum_{n\in\ZZ}|M_{\rm HL}f_n|^2\Big)^{1/2}\Big\|_{L^p(\RR,w{\rm d}x)}\leq C_{p,w}\Big\|\Big(\sum_{n\in\ZZ}|f_n|^2\Big)^{1/2}\Big\|_{L^p(\RR,w{\rm d}x)}
\end{equation*}
was proven by Andersen and John \cite[Theorem 3.1]{AndersenJohn80} and independently by Kokilashvili \cite{Kokilasvili}. The vector-valued estimate for the Hardy operator $H$,
%applied to the linear version of $H$ without the absolute value in the definition.
\begin{equation*}
\Big\|\Big(\sum_{n\in\ZZ}|Hf_n|^2\Big)^{1/2}\Big\|_{L^p(\RR,w{\rm d}x)}\leq C_{p,w}\Big\|\Big(\sum_{n\in\ZZ}|f_n|^2\Big)^{1/2}\Big\|_{L^p(\RR,w{\rm d}x)}
\end{equation*}
follows by the weighted estimate \eqref{eq:6:weight} and the general $\ell^2$-extension theorem \cite[Theorem 5.5.1]{Grafakos}, applied to the linear version of $H$ without the absolute value in the definition. By the same argument, we get that for the Hilbert transform
\begin{equation*}
    \calH f(x):={\rm p.v.}\int_{\RR}\frac{f(x-y)}{y}{\rm d}y,\quad x\in\RR,
\end{equation*}
the vector-valued estimate
\begin{equation*}
\Big\|\Big(\sum_{n\in\ZZ}|\calH f_n|^2\Big)^{1/2}\Big\|_{L^p(\RR,w{\rm d}x)}\leq C_{p,w}\Big\|\Big(\sum_{n\in\ZZ}|f_n|^2\Big)^{1/2}\Big\|_{L^p(\RR,w{\rm d}x)}
\end{equation*}
holds as well. Then the vector-valued estimate for the maximal Hilbert transform $\calH^\ast$ follows by Cotlar's inequality \cite[Theorem 5.3.4]{Grafakos}
\begin{equation*}
    \calH^\ast f(x)\lesssim M_{\rm HL}(\calH f)(x)+ M_{\rm HL}f(x),\quad x\in \RR.
\end{equation*}
Finally, the vector-valued estimate for the Carleson operator \eqref{eq:prop1.8proof} was proved by Grafakos, Martell and Soria \cite{GMS}. Now, considering the above estimates and the fact that $w \in A_\alpha^p$, the estimate \eqref{eq9:vector-va} follows.

A straightforward consequence of \eqref{eq9:vector-va} is the $\ell^2$-extension of the weighted Kanjin--Prestini estimate \eqref{wKPe}.
\begin{corollary}\label{cor:hankelsum}
Let $\alpha\geq-1/2$ and let $p\in(1,\infty)$. Let $w\colon\RR\to[0,\infty)$ be an even weight from $A_p^\alpha$. There exists a constant $C_{p,\alpha,w}>0$ such that for any sequence of functions from $L^p(\RR_+,wx^{2\alpha+1}{\rm d}x)$ we have 
\begin{equation*}
    \Big\|\Big(\sum_{n\in\ZZ}\sup_{t>0} |\widetilde{\calS}_{t}^\alpha f_n|^2\Big)^{1/2}\Big\|_{L^p(\RR_+,wx^{2\alpha+1}{\rm d}x)}\leq C_{p,\alpha,w}\Big\|\Big(\sum_{n\in\ZZ}|f_n|^2\Big)^{1/2}\Big\|_{L^p(\RR_+,wx^{2\alpha+1}{\rm d}x)}.
\end{equation*}
\end{corollary}
\begin{proof}
If $g\in L^p(\RR_+,wx^{2\alpha+1}{\rm d}x)$, then we can define an even extension of $g$ as a function on $L^p(\RR,w|x|^{2\alpha+1}{\rm d}x)$ by
\begin{equation*}
    \overline{g}(x)=\begin{cases}
        g(x),\text{ for } x\geq 0,\\
        g(-x),\text{ for } x<0.
    \end{cases}
\end{equation*}
Then, the proof straightforwardly follows from the equality
\begin{equation*}
    {\calS}_{t}^\alpha \overline{g}(x)=\widetilde{\calS}_{t}^\alpha \overline{g}_{\rm e}(|x|)+x\widetilde{\calS}_{t}^{\alpha+1}\big(\overline{g}_{\rm o}/{(\cdot)}\big)(|x|),\quad x\in\RR,
\end{equation*}
and the vector-valued estimate \eqref{eq9:vector-va}.
\end{proof}
The above result appears to be entirely novel, as we have found no similar results in the literature. Additionally, it has implications for the Carleson operator on $\mathbb{R}^n$.
\begin{theorem}[Weighted vector-valued estimate for the Carleson operator on radial functions]\label{thm:carlesonvectorradial}
Let \( n \geq 1 \) and let \( p \in (1, \infty) \). Assume that \( w \in A_p^\alpha \), with \( \alpha = (n - 2)/2 \), is an even weight. We consider its radial extension \( w(|\cdot|) \), which will also be denoted by \( w \). Let $L_{\rm rad}^p(\RR^n,w{\rm d}x)$ denote the space of radial functions in $L^p(\RR^n,w{\rm d}x)$ and let
\begin{equation*}
    \calS_t f(x):=\calF^{-1}(\mathds{1}_{B(0,t)}\calF f)(x)
\end{equation*}
denote the partial sum of the Fourier integrals. Then the vector-valued estimate holds
\begin{equation}\label{eq:hankelsumtoprove}
    \Big\|\Big(\sum_{m\in\ZZ}\sup_{t>0} |\calS_t f_m|^2\Big)^{1/2}\Big\|_{L^p(\RR^n,w{\rm d}x)}\lesssim_{n,w}\Big\|\Big(\sum_{m\in\ZZ}|f_m|^2\Big)^{1/2}\Big\|_{L^p(\RR^n,w{\rm d}x)}.
\end{equation}
for $(f_m)_{m\in\ZZ}\in L_{\rm rad}^p(\RR^n;\ell^2(\ZZ),w{\rm d}x)$. In particular, when $w=1$, then for
\begin{equation*}
    \frac{2n}{n+1}<p<\frac{2n}{n-1}
\end{equation*}
one has
\begin{equation*}
    \Big\|\Big(\sum_{m\in\ZZ}\sup_{t>0} |\calS_t f_m|^2\Big)^{1/2}\Big\|_{L^p(\RR^n)}\lesssim_{n}\Big\|\Big(\sum_{m\in\ZZ}|f_m|^2\Big)^{1/2}\Big\|_{L^p(\RR^n)}
\end{equation*}
for any sequence of radial functions $(f_m)_{m\in\ZZ}\in L_{\rm rad}^p(\RR^n;\ell^2(\ZZ))$.
\end{theorem}
\begin{proof}
Let $w\in A_p^\alpha$ with $\alpha=(n-2)/2$ and let $g\in L^p(\RR^n)$ be a radial function, that is $g(x)=g_0(|x|)$, for some $g_0\colon[0,\infty)\to\RR$. Then, by using spherical coordinates, we have
\begin{equation}\label{eq:p1}
    \|g\|_{L^p(\RR^n,w{\rm d}x)}=\Big(\int_{\mathbb{S}^{n-1}}\int_{0}^\infty|g_0(r)|^p w(r)r^{n-1}{\rm d}r{\rm d}\sigma(\theta)\Big)^{1/p}=c_{n,p}\Big(\int_{0}^\infty|g_0(r)|^p w(r)r^{n-1}{\rm d}r\Big)^{1/p},
\end{equation}
where $\sigma$ is the spherical measure on $\mathbb{S}^{n-1}$, and $c_{n,p}>0$ is a constant independent of $g\in L^p(\RR^n, w{\rm d}x)$. Moreover, it is well known that the Fourier transform of the radial function $g(x)=g_0(|x|)$ is radial and may be expressed in terms of the Hankel transform of order $\alpha=(n-2)/2$. That is, one has
\begin{equation*}
    \calF g(x)=(2\pi)^{n/2}\calH_{\alpha}(g_0)(|x|),\quad x\in\RR^n.
\end{equation*}
Consequently, we may write
\begin{equation*}
    \calS_t g(x)=(2\pi)^n\calH_\alpha\big(\mathds{1}_{[0,t]}\calH_\alpha( g_0)\big)(|x|),\quad x\in\RR^n.
\end{equation*}
Now, let $(f_m)_{m\in\ZZ}\subset C_{\rm c}^\infty(\RR^n)$ be the sequence of radial functions. Denote by $f_{m,0}$ the radial part of $f_m$, that is, $f_m(x)=f_{m,0}(|x|)$. Then
\begin{equation*}
    \Big(\sum_{m\in\ZZ}\sup_{t>0} |\calS_t f_m|^2\Big)^{1/2}=(2\pi)^n\Big(\sum_{m\in\ZZ}\sup_{t>0} \big|\calH_\alpha\big(\mathds{1}_{[0,t]}\calH_\alpha( f_{m,0})\big)\big|^2\Big)^{1/2}
\end{equation*}
and we see that it is a radial function. By \eqref{eq:p1} we have 
\begin{align*}
    \Big\|\Big(\sum_{m\in\ZZ}\sup_{t>0} |\calS_t f_m|^2\Big)^{1/2}\Big\|_{L^p(\RR^n,w{\rm d}x)}&=c_{n,p}(2\pi)^n\Big\|\Big(\sum_{m\in\ZZ}|\widetilde{\calS}_{t}^\alpha f_{m,0}|^2\Big)^{1/2}\Big\|_{L^p(\RR_+,wx^{2\alpha+1}{\rm d}x)},
\end{align*}
with $\alpha=(n-2)/2$. Since $w\in A_p^\alpha$, by Corollary~\ref{cor:hankelsum}, the latter is bounded by
\begin{equation*}
\Big\|\Big(\sum_{n\in\ZZ}|f_{m,0}|^2\Big)^{1/2}\Big\|_{L^p(\RR_+,w|x|^{2\alpha+1}{\rm d}x)}\lesssim_n\Big\|\Big(\sum_{m\in\ZZ}|f_m|^2\Big)^{1/2}\Big\|_{L^p(\RR^n,w{\rm d}x)},
\end{equation*}
where the last inequality follows by \eqref{eq:p1}. This ends the proof of \eqref{eq:hankelsumtoprove}.
\end{proof}

\subsection{Rubio de Francia estimate \eqref{RdF eq:1} and transference of radial multipliers}
To establish \eqref{RdF eq:1}, we use the ideas presented by Stempak and Trebels \cite{StempakTrebels1997} and show that we can derive estimates for $\alpha > -1/2$ from the estimate for $\alpha = -1/2$. In order to do so, we use the so-called transplantation inequalities for the modified Dunkl transform. Nowak and Stempak \cite{NowakStempak} introduced the Dunkl transplantation operator
\begin{equation*}
    T_{\alpha\gamma}:=\widecheck{D}_\alpha\circ D_\gamma,
\end{equation*}
for $\alpha,\gamma\geq-1/2$, $\alpha\neq\gamma$. Here we remind that $D_\alpha$ is the modified Dunkl transform \eqref{modedunkl}. Analogously, one may define the Hankel transplantation by setting
\begin{equation*}
    \widetilde{T}_{\alpha\gamma}:=H_\alpha\circ H_\gamma,
\end{equation*}
for the same range of $\alpha$ and $\gamma$ as before. Here $H_\alpha$ is the modified Hankel transform \eqref{modHankel}. The Hankel transplantation operators were introduced by Guy \cite{Guy}, who proved the first transplantation inequality.

Let $w\colon\RR\to[0,\infty)$ be an even non-negative weight. We are interested whenever the transplantation estimate
\begin{equation*}
    \big\|T_{\alpha\gamma}f\|_{L^p(\RR,w{\rm d}x)}\lesssim_{p,\alpha,\gamma, w}\|f\|_{L^p(\RR,w{\rm d}x)}
\end{equation*}
holds. Nowak and Stempak \cite[Proposition 2.1]{NowakStempak} observed that the above inequality holds if and only if 
\begin{equation*}
    \big\| \widetilde{T}_{\alpha\gamma}f\|_{L^p(\RR_+,w{\rm d}x)}\lesssim_{p,\alpha,\gamma, w}\|f\|_{L^p(\RR_+,w{\rm d}x)}
\end{equation*}
and if the analogous inequality with $\alpha$ and $\gamma$ replaced by $\alpha+1$ and $\gamma+1$ holds. In our case, we will be interested in a particular type of weights, namely power weights given by $w(x)=|x|^\beta$. Therefore, we want to know whenever
\begin{equation}\label{RdF eq:2}
    \big\|T_{\alpha\gamma} f\big\|_{L^p(\RR,|x|^\beta{\rm d}x)}\leq C_{p,\alpha,\gamma,\beta}\|f\|_{L^p(\RR,|x|^\beta{\rm d}x)}
\end{equation}
holds. From the above discussion, this problem can be reduced to establishing the transplantation inequalities for the Hankel transplantation operator with parameters $\alpha, \gamma$ and $\alpha+1, \gamma+1$. Stempak \cite[Collorary 1.4]{Stempak2002} showed that the inequality 
\begin{equation*}
    \big\| \widetilde{T}_{\alpha\gamma}f\|_{L^p(\RR_+,|x|^\beta{\rm d}x)}\lesssim_{p,\alpha,\gamma, \beta}\|f\|_{L^p(\RR_+,|x|^\beta{\rm d}x)}
\end{equation*}
holds if one has
\begin{equation}\label{RdF eq:3}
-p(\alpha+1/2)-1<\beta<p(\gamma+3/2)-1.
\end{equation}
Consequently, the estimate \eqref{RdF eq:2} holds if condition \eqref{RdF eq:3} is satisfied. Since $T_{\alpha\gamma}$ is a linear operator, by the $\ell^2$-extension theorem \cite[Theorem 5.5.1]{Grafakos}, we obtain the following vector-valued estimate
\begin{equation}\label{VVTRANS:eq4}
\Big\|\Big(\sum_{n\in\ZZ}|T_{\alpha\gamma}f_n|^2\Big)^{1/2}\Big\|_{L^p(\RR,|x|^\beta{\rm d}x)}\leq C_{p,\alpha,\gamma, \beta}\Big\|\Big(\sum_{n\in\ZZ}|f_n|^2\Big)^{1/2}\Big\|_{L^p(\RR,|x|^\beta{\rm d}x)}
\end{equation}
holds if the condition \eqref{RdF eq:3} is satisfied.

The next result is a technical tool which guarantees that we may work with smooth compactly supported functions only. The proof is easy and is based on the similar result proven by Stempak and Trebels \cite{StempakTrebels1997} in the context of the Hankel transform.
\begin{fact}
Assume $p\in(1,\infty)$ and $\beta>-1-p(\alpha+1/2)$. Then the set $D_\alpha(C_{\rm c}^\infty(\RR\setminus\{0\}))$ is dense in $L^p(\RR, |x|^\beta{\rm d}x)$.
\end{fact}
\begin{proof}
Fix $\varepsilon>0$. Consider $f \in L^p(\RR, |x|^\beta{\rm d}x)$ and split it into its even and odd components, $f = f_{\rm e} + f_{\rm o}$. These functions are uniquely defined by their values on the positive real numbers. By \cite[Corollary 4.8]{StempakTrebels1997} there are functions $g,h \in C_{\rm c}^\infty(\RR_+)$ such that
\begin{align*}
    \big\|f_{\rm e} - H_\alpha(g)\big\|_{L^p(\RR_+,x^\beta{\rm d}x)}<\varepsilon\quad\text{and}\quad \big\|f_{\rm o} - H_{\alpha+1}(h)\big\|_{L^p(\RR_+,x^\beta{\rm d}x)}<\varepsilon.
\end{align*}
Now, let us define a new function on $\RR\setminus\{0\}$ as
\begin{equation*}
    \widetilde{f}(x):=g(|x|)+i\sgn(x)h(|x|), \quad x\in\RR\setminus\{0\}.
\end{equation*}
Clearly, $\widetilde{f}\in C_{\rm c}^\infty(\RR\setminus\{0\})$ and it is easy to see that one has
\begin{equation*}
    \widetilde{f}_{\rm e}(x)=g(x)\quad\text{and}\quad \widetilde{f}_{\rm o}(x)=i h(x),\quad\text{for} \quad x>0.
\end{equation*}
Moreover, by using the fact that the modified Dunkl transform can be expressed in terms of $H_\alpha$ and $H_{\alpha+1}$, we get
\begin{equation*}
    D_\alpha \widetilde{f}(x)=H_\alpha(g)(|x|)+ \sgn(x) H_{\alpha+1}(h)(|x|),\quad x\in \RR.
\end{equation*}
Consequently, we have
\begin{align*}
    \big\|f-D_\alpha\widetilde{f}\big\|_{L^p(\RR, |x|^\beta{\rm d}x)}&\leq \big\|f_{\rm e}-H_\alpha(g)(\cdot)\big\|_{L^p(\RR, |x|^\beta{\rm d}x)}+\big\|f_{\rm o}-\sgn(\cdot)H_\alpha(h)(|\cdot|)\big\|_{L^p(\RR, |x|^\beta{\rm d}x)}\\
    &\leq 2\big\|f_{\rm e} - H_\alpha(g)\big\|_{L^p(\RR_+,x^\beta{\rm d}x)}+2\big\|f_{\rm o} - H_{\alpha+1}(h)\big\|_{L^p(\RR_+,x^\beta{\rm d}x)}<4\varepsilon.
\end{align*}
This ends the proof.
\end{proof}
\begin{definition}
Let $p \in [1,\infty)$ and $\alpha \geq -1/2$. Let $(m_n)_{n\in\mathbb{Z}}$ be a sequence of bounded measurable functions on $\mathbb{R}$. We say that the sequence $(m_n)_{n\in\mathbb{Z}}$ is a vector-valued multiplier for the modified Dunkl transform $D_\alpha$ on $L^p(\RR,|x|^\beta{\rm d}x)$ if there exists a constant $C > 0$ such that
\begin{equation*}
    \Big\|\Big(\sum_{n\in\ZZ}\big|\widecheck{D}_\alpha(m_n D_\alpha f)\big|^2\Big)^{1/2}\Big\|_{L^p(\RR,|x|^\beta{\rm d}x)}\leq C\|f\|_{L^p(\RR,|x|^\beta{\rm d}x)},\quad f \in {L^p(\RR,|x|^\beta{\rm d}x)}.
\end{equation*}
The set of such multipliers will be denoted by $M_p^{\alpha,\beta}$. In a similar way, we define the vector-valued multipliers for the standard Dunkl transform $\mathcal{D}_\alpha$ on $L^p(\RR,|x|^\beta|x|^{2\alpha+1}{\rm d}x)$. The set of such multipliers will be denoted by $\mathcal{M}_p^{\alpha,\beta}$.
\end{definition}
The following result states that a multiplier for the modified Dunkl transform $D_\alpha$ is also a multiplier for $D_\gamma$, provided that the condition \eqref{CONTITION} is satisfied. A corresponding result for scalar multipliers for the modified Hankel transform was proven by Stempak and Trebels \cite[Corollary 2.1]{StempakTrebels1997}.

\begin{proposition}\label{prop:multi}
Let $\alpha,\gamma\geq-1/2$ and let $\beta\in\RR$ satisfy the following condition
\begin{equation}\label{CONTITION}
    -1-p\min{\big\{\alpha+1/2,\gamma+1/2\big\}}<\beta<-1+p\min{\big\{\alpha+3/2,\gamma+3/2\big\}}.
\end{equation}
Then
\begin{equation*}
    M_p^{\gamma,\beta}=M_p^{\alpha,\beta}.
\end{equation*}
\end{proposition}
\begin{proof}
Let $f\in C_{\rm c}^\infty(\RR\setminus\{0\})$. Assume that $(m_n)_{n\in\ZZ}\in M_p^{\gamma,\beta}$.
Let us note that $\widecheck{D}_\alpha(m_n D_\alpha f)\in L^2(\RR)$, hence we may write that
\begin{equation*}
    \widecheck{D}_\alpha(m_n D_\alpha f)=T_{\alpha\gamma}\widecheck{D}_\gamma(m_n D_\alpha f).
\end{equation*}
As a result we my apply \eqref{VVTRANS:eq4} to get
\begin{equation*}
    \Big\|\Big(\sum_{n\in\ZZ}\big|\widecheck{D}_\alpha(m_n D_\alpha f)\big|^2\Big)^{1/2}\Big\|_{L^p(\RR,|x|^\beta{\rm d}x)}\leq C_{p,\alpha,\gamma,\beta} \Big\|\Big(\sum_{n\in\ZZ}\big|\widecheck{D}_\gamma(m_n D_\alpha f)\big|^2\Big)^{1/2}\Big\|_{L^p(\RR,|x|^\beta{\rm d}x)}
\end{equation*}
since the condition \eqref{CONTITION} is satisfied. Now, by using the fact that $D_\gamma(\widecheck{D}_\gamma g)=g$ we may use the fact that $(m_n)_{n\in\ZZ}\in M_p^{\gamma,\beta}$ to estimate the above by
\begin{equation*}
    \Big\|\Big(\sum_{n\in\ZZ}\big|\widecheck{D}_\gamma(m_n D_\gamma(\widecheck{D}_\gamma(D_\alpha f)\big|^2\Big)^{1/2}\Big\|_{L^p(\RR,|x|^\beta{\rm d}x)}\lesssim_{p,m,\gamma,\beta}\big\|\widecheck{D}_\gamma(D_\alpha f)\big\|_{L^p(\RR,|x|^\beta{\rm d}x)}\lesssim\|f\|_{L^p(\RR,|x|^\beta{\rm d}x)}.
\end{equation*}
This shows $M_p^{\gamma,\beta}\subseteq M_p^{\alpha,\beta}$. To obtain the second part, we exchange the roles of $\alpha$ and $\gamma$.
\end{proof}
The next fact allows us to connect the multipliers for the modified Dunkl transform $D_\alpha$ with the multipliers for the standard Dunkl transform $\mathcal{D}_\alpha$.

\begin{fact}\label{fact:multi}
Let $\alpha\geq -1/2$. Then $(m_n)_{n\in\mathbb{Z}}\in M_{p}^{\alpha,\beta}$ if and only if $(m_n)_{n\in\mathbb{Z}}\in\mathcal{M}_{p}^{\alpha,\beta^*}$ where $\beta^*=\beta-(\alpha+1/2)(2-p)$.
\end{fact}
\begin{proof}
The proof uses the relation \eqref{dunksrelation}. We repeat the steps used during the proof of \cite[Proposition 3.4]{NowakStempak}. We omit the details.
\end{proof}

\begin{corollary}\label{eqmulti}
Let $\alpha,\gamma\geq -1/2$ and let $\beta\in\RR$ satisfy condition \eqref{CONTITION}. Then 
\begin{equation*}
    {\calM}_p^{\alpha,\beta}={\calM}_p^{\gamma,\beta^\ast}.
\end{equation*}
In particular, when \( \alpha = -1/2 \), then \( \beta^\ast = \beta - (1/2 + \gamma)(2 - p) \), and we have
\begin{equation*}
    \mathcal{M}_p^{-1/2, \beta} = \mathcal{M}_p^{\gamma, \beta^\ast}.
\end{equation*}
\end{corollary}
\begin{proof}
The proof easily follows from Proposition~\ref{prop:multi} and Fact~\ref{fact:multi}.
\end{proof}
In other words, the above corollary implies that for $-1<\beta<p-1$, the inequality
\begin{equation*}
    \Big\|\Big(\sum_{n\in\ZZ}\big|\calF^{-1}(m_n \calF f)|^2\Big)^{1/2}\Big\|_{L^p(\RR,|x|^\beta{\rm d}x)}\leq C_{p,m,\beta}\|f\|_{L^p(\RR,|x|^\beta{\rm d}x)}
\end{equation*}
implies
\begin{equation}
    \Big\|\Big(\sum_{n\in\ZZ}|\widecheck{\calD}_\alpha(m_n \calD_\alpha f)|^2\Big)^{1/2}\Big\|_{L^p(\RR,|x|^{\beta^\ast+2\alpha+1}{\rm d}x)}\lesssim_{p,\alpha,\beta} C_{p,m,\beta}\|f\|_{L^p(\RR,|x|^{\beta^\ast+2\alpha+1}{\rm d}x)},
\end{equation}
where $\beta^\ast=\beta-(1/2+\alpha)(2-p)$. Now, we know that the Rubio de Francia estimate \eqref{RdF eq:1} holds when $\alpha=-1/2$ and $w\in A_{p/2}$ . Consequently, we get that for $w(x)=|x|^\beta$ the inequality \eqref{RdF eq:1} holds if
\begin{equation*}
    -1<\beta^\ast+(1/2+\alpha)(2-p)<p/2-1.
\end{equation*}
In the case of dyadic intervals $I_j$ we know that the weighed Littlewood--Paley theory holds for $w\in A_{p}$, which in the case $w(x)=|x|^\beta$ imply
\begin{equation*}
    -1<\beta^\ast+(1/2+\alpha)(2-p)<p-1.
\end{equation*}
Finally, for $p=2$ and $\alpha=-1/2$, we know that \eqref{RdF eq:1} holds when $-1<\beta
\leq0$. Which ends the proof of \eqref{RdF eq:1} and consequently, the proof of Theorem~\ref{main:oscil:dunkl}.

Now we give the proof of Theorem~\ref{thm:RdFR}.
\begin{proof}[Proof of Theorem~\ref{thm:RdFR}]
If each $m_n$ is an even function, then by \cite[Lemma 3.7]{Slom} we get that
\begin{equation*}
    \widecheck{\calD}_\alpha(m_n \calD_\alpha f)=\calH_\alpha(m_n\calH_\alpha(f_{\rm e}))+ \sgn(x)\calH_{\alpha+1}(m_n\calH_{\alpha+1}(f_{\rm o})).
\end{equation*}
Now, if we assume that
\begin{equation*}
    \Big\|\Big(\sum_{n\in\ZZ}\big|\calF^{-1}(m_n\calF f)\big|^2\Big)^{1/2}\Big\|_{L^p(\RR,|x|^\beta{\rm d}x)}\lesssim\|f\|_{L^p(\RR,|x|^\beta {\rm d}x)}.
\end{equation*}
Then by Corollary~\ref{eqmulti} this implies 
\begin{equation}\label{eq:RdFD}
    \Big\|\Big(\sum_{n\in\ZZ}|\widecheck{\calD}_\alpha(m_n \calD_\alpha f)|^2\Big)^{1/2}\Big\|_{L^p(\RR,|x|^{\beta^\ast+2\alpha+1}{\rm d}x)}\lesssim_{p,\alpha,\beta,m} \|f\|_{L^p(\RR,|x|^{\beta^\ast+2\alpha+1}{\rm d}x)}
\end{equation}
for any $f\in L^p(\RR,|x|^{\beta^\ast+2\alpha+1}{\rm d}x)$. Now, if we restrict ourselves to even functions in \eqref{eq:RdFD}, we get
\begin{equation*}
    \Big\|\Big(\sum_{n\in\ZZ}|\calH_\alpha(m\calH_\alpha f)|^2\Big)^{1/2}\Big\|_{L^p(\RR_+,x^{\beta^\ast+2\alpha+1}{\rm d}x)}\lesssim_{p,\alpha,\beta,m}\|f\|_{L^p(\RR_+,x^{\beta^\ast+2\alpha+1}{\rm d}x)}.
\end{equation*}
Let $\alpha=(n-2)/2$. We proceed as at the end of the proof of Theorem~\eqref{thm:carlesonvectorradial} to show that
\begin{equation*}
    \Big\|\Big(\sum_{n\in\ZZ}\big|\calF^{-1}(m_n\calF f)\big|^2\Big)^{1/2}\Big\|_{L^p(\RR^n,|x|^{\beta^*}{\rm d}x)}\lesssim\|f\|_{L^p(\RR^n,|x|^{\beta^*} {\rm d}x)}.
\end{equation*}
This ends the proof the proof of Theorem~\ref{thm:RdFR}.
\end{proof}

\end{document}